\documentclass{article}
\usepackage{amssymb,amsmath,amsthm, mathrsfs}
\usepackage{graphicx, float,color}
\usepackage[encapsulated]{CJK}
\usepackage[english]{babel}
\usepackage[colorlinks=true]{hyperref}
\hypersetup{urlcolor=blue, linkcolor=red, citecolor=blue}
\usepackage{geometry} % to change the page dimensions
\usepackage{array} % for better arrays (eg matrices) in maths
\usepackage{paralist} % very flexible & customisable lists (eg. enumerate/itemize, etc.)
\usepackage{verbatim} % adds environment for commenting out blocks of text & for better verbatim
\usepackage[all]{xy}
\newtheorem{theorem}{Theorem}[section]
\newtheorem {lemma}[theorem]{Lemma}
\newtheorem {proposition}[theorem]{Proposition}
\newtheorem {corollary}[theorem]{Corollary}

\newtheorem {example}[theorem]{Example}
\newtheorem {remark}[theorem]{Remark}
\newtheorem {question}[theorem]{Question}
\usepackage{sectsty}
\allsectionsfont{\sffamily\mdseries\upshape} % (See the fntguide.pdf for font help)

\title{A note on disjoint hypercyclicity for invertible bilateral pseudo-shifts on $\ell^{p}(\mathbb{Z})$}
\author{Song-Ung Ri, Hyon-Hui Ju*, Jin-Myong Kim}

\date{}
\begin{document}
\maketitle
{\makeatletter\renewcommand*\@makefnmark{}\footnotetext{

*Corresponding author E-mail address:
hh.ju0524@ryongnamsan.edu.kp(HyonHui Ju) }\makeatother}

\centerline{Faculty of Mathematics, {\bf Kim Il Sung} University,}

\centerline{Pyongyang, Democratic People's Republic of Korea}
%\selectlanguage{english} %%% remove comment delimiter ('%') and select language if required

\begin{abstract}
 We give a sufficient condition for bilateral pseudo-shifts on $\ell^{p}(\mathbb{Z})$ to be disjoint hypercyclic and concrete examples. Next we partially answer the open problem posed by Martin, Menet and Puig (Disjoint frequently hypercyclic pseudo-shifts, J. Funct. Anal., \textbf{283} (1), 109474, 2022)\cite{MMP22} concerned with disjoint reiterative hypercyclicity, that is, we show that  as for the operators on a reflexive Banach space, reiteratively hypercyclic ones are disjoint hypercyclic if and only if they are disjoint reiteratively hypercyclic.
\end{abstract}



\section{Introduction}

Pseudo-shifts is a generalization of weighted shifts which play an important role in linear dynamics  offering many examples and counterexamples in this field. Since the role of weighted shifts is very limited in joint dynamics (in fact, they can never be disjoint weakly mixing and, thus, never be disjoint mixing nor satisfy the Disjoint Hypercyclicity Criterion; see \cite{BMS14}), several authors studied joint dynamics of pseudo-shifts which can support a wider range of disjoint dynamical properties than the weighted shifts.

\c{C}olako\u{g}lu,  Martin and Sanders \cite{CMS20}  characterized disjoint hypercyclicity and disjoint weak mixing for unilateral pseudo-shifts on $\ell^{p}(\mathbb{N})$ with $1\leq p <\infty$. Martin, Menet and  Puig \cite{MMP22}  obtained a Disjoint Frequent Hypercyclicity Criterion and showed that it characterizes disjoint frequent hypercyclicity for a family of unilateral pseudo-shifts on $c_0(\mathbb{N})$ and $\ell^{p}(\mathbb{N})$, $1\leq p <\infty$.

On the other hand, we can also find a limited role of weighted shift operators in joint dynamics in a result of  B\`{e}s, Martin and  Sanders \cite{BMS14}. In \cite{BMS14}, it was proved that for any tuple of bilateral weighted shifts on $\ell^{p}(\mathbb{Z})$ with $1\leq p <\infty$, whenever one of the tuple is invertible, the tuple fails to be disjoint hypercyclic.
Nevertheless B\`{e}s and Peris \cite{BP07} provided an equivalent condition for a tuple of powers of bilateral weighted shifts on $\ell^{p}(\mathbb{Z})$ with $1\leq p <\infty$ to be disjoint hypercyclic. As powers of weighted shifts are pseudo-shifts, this means that as for the pseudo-shifts, there exist some tuples of invertible ones on $\ell^{p}(\mathbb{Z})$ with $1\leq p <\infty$, such that they are disjoint hypercyclic.
 
Then, it is natural to ask what the conditions need for the tuple of  bilateral pseudo-shifts on $\ell^{p}(\mathbb{Z})$ to be disjoint hypercyclic. 

In this note we give a sufficient condition for bilateral pseudo-shifts on $\ell^{p}(\mathbb{Z})$ with $1\leq p <\infty$ to be disjoint hypercyclic. Using this condition, we provide some examples of invertible pseudo-shifts which are disjoint hypercyclic, especially an example which one of them can not be described as powers of weighted shift, but they are disjoint hypercyclic

At the end of this note we discuss the open problem from Martin, Menet and  Puig \cite{MMP22}.
In \cite{MMP22}, they showed that two reiteratively hypercyclic unilateral pseudo-shifts  on $\ell^{p}(\mathbb{N})$ with the same inducing map are disjoint hypercyclic if and only if they are disjoint reiteratively hypercyclic, and then they asked:

\begin{question}\label{qest1.1}\textnormal{(Question 5.1 of \cite{MMP22})
Is it possible to find two reiteratively hypercyclic operators $T_{1}$, $T_{2}$ which are disjoint hypercyclic but not disjoint reiteratively hypercyclic? Can $T_{1}$ and $T_{2}$ be taken as unilateral weighted shifts on $c_{0}(\mathbb{N})$?}
\end{question}

We show that on a reflexive Banach space it is impossible, that is, reiteratively hypercyclic operators on a reflexive Banach space are disjoint hypercyclic if and only if they are disjoint reiteratively hypercyclic.

The organization of the paper is as follows.
In Section 2, we give a sufficient condition for pseudo-shifts on $\ell^{p}(\mathbb{Z})$ with $1\leq p <\infty$ to be disjoint hypercyclic and concrete examples. In Section 3, we discuss Question \ref{qest1.1}.

In the rest of Introduction, let us fix some notations and terminology.

Let $X$ denote a separable infinite-dimensional Banach space over the real or complex scalar field $\mathbb{K}$ and $\mathcal{L}(X)$ denote the algebra of continuous linear operators on $X$. We say that an operator $T$ in $\mathcal{L}(X)$ is \textit{hypercyclic} if there exists a vector $x \in X$ (also called \textit{hypercyclic}) such that its orbit $ Orb(x, T) =\{T^{n}x:n\in {\mathbb N} \}$ is dense in $X$. Equivalently, there exists a vector $x \in X$ such that for any non-empty open set $U \subset X$ the return set $N(x,U) := \{n \in {\mathbb{N}}:T^{n}x \in U \}$ is non-empty. We will denote by $HC(T)$ the set of hypercyclic vectors for $T$.

Disjointness in hypercyclicity is introduced independently by Bernal \cite{B07} and by B\`{e}s and Peris \cite{BP07}     in 2007.
For $N\geq 2$, operators $T_{1}$,...,$T_{N}\in \mathcal{L}(X)$ are called \textit{disjoint hypercyclic} if the direct sum operator $T_{1}\oplus \cdots \oplus T_{N}$ has a hypercyclic vector of the form $(x,...,x)\in X^{N}$. Such a vector $x\in X$ is called a \textit{disjoint hypercyclic vector} for $T_{1},...,T_{N}$.

And recall the notions of reiterative recurrence, reiterative hypercyclicity and disjoint reiterative hypercyclicity.
The \textit{upper Banach density} of a set $A\subset \mathbb{N}$ is defined by
\[\overline{Bd} (A)=\lim_{N \rightarrow \infty}\limsup_{m \rightarrow \infty}\frac{\# (A\cap [m+1, m+N])}{N}\]
where $\# (\cdot)$ means the cardinal number of $(\cdot)$. A vector $x\in X$ is said to be \textit{reiteratively recurrent} (resp. \textit{reiteratively hypercyclic}) if $\overline{Bd}(N(x,U))>0$ for every neighborhood $U$ of $x$ (resp. for every non-empty open subset $U$ of X). We will denote by $RRec(T)$ (resp. $RHC(T)$) the set of such points and we say that $T$ is \textit{reiteratively recurrent} (resp. \textit{reiteratively hypercyclic}) if $RRec(T)$ is dense in $X$ (resp. if $RHC(T)\neq \emptyset$).
And we say that $T_{1}$,...,$T_{N}$ are \textit{disjoint reiteratively hypercyclic} if the direct sum operator $T_{1}\oplus \cdots \oplus T_{N}$ has a reiteratively hypercyclic vector of the form $(x,...,x)\in X^{N}$. Such a vector $x\in X$ is called $disjoint$ $reiteratively$ $hypercyclic$.

Recall that for a bounded weight sequence $(w_{n})_{n\in \mathbb{N}}$ (resp.$(w_{n})_{n\in \mathbb{Z}}$) of nonzero scalars, the unilateral (resp. bilateral) weighted shift $B_{w}$ on $\ell^{p}(\mathbb{N})$ (resp. $\ell^{p}(\mathbb{Z})$) is defined by $B_{w}e_{1}=0$ and $B_{w}(e_{n})=w_{n}e_{n-1}$ for $n\geq 2$ (reap. $B_{w}(e_{n})=w_{n}e_{n-1}$ for $n\in \mathbb{Z}$), where $(e_{n})_{n}$ is the canonical basis.
Grosse-Erdmann \cite{G00} defined pseudo-shifts as a generalization of weighted shifts. We fix the definition of bilateral pseudo-shifts according to \cite{G00} as follows:\\

\noindent\textbf{Definition 1.3.}
Let $X=\ell^{p}(\mathbb{Z})$ with $1\leq p< \infty$, $w=(w_{i})_{i\in \mathbb{Z}}$ be a bounded, nonzero bilateral weight sequence, and let $f : \mathbb{Z}\rightarrow \mathbb{Z}$ be an invertible map. The \emph{bilateral pseudo-shift} $T_{f,w}:X \rightarrow X$ induced by the \emph{inducing map} $f$ and the weight sequence $w=(w_{i})_{i\in \mathbb{Z}}$, is the linear operator given by

\[T_{f,w}(\sum^{\infty}_{-\infty}x_{j}e_{j})=\sum^{\infty}_{-\infty}w_{f_{(j)}}x_{f_{(j)}}e_{j},\]
where $(e_{j})_{j\in \mathbb{Z}}$ is the canonical basis of $X$.\\

Bilateral weighted shifts are bilateral pseudo-shifts with the inducing map $f(n)=n+1$ with $n\in \mathbb{Z}$ and a weighted shift raised to the power $r\in \mathbb{N}$ is a pseudo-shift with the inducing map $f(n)=n+r$ with $n\in \mathbb{Z}$.
Obviously, a bilateral pseudo-shift is invertible if and only if $\inf_{i\in \mathbb{Z}}|w_{i}|>0$ and $\sup_{i\in \mathbb{Z}}|w_{i}|<\infty$.

\section{Disjoint hypercyclic bilateral pseudo-shifts}

In this section, we first provide a sufficient condition for bilateral pseudo-shifts to be disjoint hypercyclic. Then we give some examples of invertible bilateral pseudo-shifts which are disjoint hypercyclic. 

Strong Disjoint Blow-up/Collapse Property gives us an useful way to show that operators are disjoint hypercyclic. We say the operators $T_{1}$,...,$T_{N}\in  \mathcal{L}(X)$ with $N\geq 2$ satisfy the $Strong$ $Disjoint$ $Blow-up/Collapse$ $Property$ if for each integer $L\in \mathbb{N}$ and for any non-empty open sets $W$, $U_{1-L}$, ... , $U_{0}$, $U_{1}$, ... ,$U_{N}$ of $X$ with $0\in W$, there exists an integer $n\in \mathbb{N}$ such that
$W\cap T^{-n}_{1}(U_{1})\cap \cdots \cap T^{-n}_{N}(U_{N})\neq \emptyset$ and $U_{l}\cap T^{-n}_{1}(W)\cap \cdots \cap T^{-n}_{N}(W)\neq \emptyset$  for integers $1-L\leq l \leq 0$.
It was checked that if the operators $T_{1}$,...,$T_{N}\in  \mathcal{L}(X)$ satisfy Strong Disjoint Blow-up/Collapse Property then they are disjoint hypercyclic by Salas \cite{S13}.
\c{C}olako\u{g}lu,  Martin and Sanders \cite{CMS20} provided a sufficient condition for operators to satisfy Strong Disjoint Blow-up/Collapse Property as follows.

\begin{theorem}{\label{them2.1}}\textnormal{(Theorem 2.2 of \cite{CMS20}) (Disjoint Blow-up/Collapse Criterion)}
For operators $T_{1}$,...,$T_{N}\in  \mathcal{L}(X)$ with $N\geq 2$, suppose there exist a strictly increasing sequence $(n_{k})$ of positive integers, a dense subset $X_{0}$ of $X$ and maps $S_{k}:\oplus^{N}_{i=1}X_{0}\rightarrow X$ which satisfy the following:

$(i)$ for each vector $x\in X_{0}$ and integer $i$ with $1\leq i\leq N$, we have $T^{n_{k}}_{i}x\rightarrow 0$ as $k\rightarrow \infty$,

$(ii)$ for each $\varepsilon > 0$, integer $k\in \mathbb{N}$ and vectors $x_{1}$,..., $x_{N}\in X_{0}$, there exists an integer $k\geq K$ such that

$\;\;$ $(a)$ $\|S_{k}(x_{1},...,x_{N})\|<\varepsilon$ and

$\;\;$ $(b)$ $\|T^{n_{k}}_{i}S_{k}(x_{1},...,x_{N})-x_{i}\|<\varepsilon$, for integers $1\leq i\leq N$.

Then the operators $T_{1}$,...,$T_{N}$ satisfy the Strong Disjoint Blow-up/Collapse Property, and hence possess a dense disjoint hypercyclic manifold.
\end{theorem}

The Disjoint Blow-up/Collapse Criterion has an equivalent condition as follows.

\begin{remark}{\label{rmk2.2}}\textnormal{(Remark 2.3 of \cite{CMS20})
The operators $T_{1},...,T_{N}\in \mathcal{L}(X)$ with $n\geq 2$ satisfy the Disjoint Blow-up/Collapse Criterion if and only if there is a dense set $X_{0}$ such that for each $\varepsilon> 0$, integers $K,L\in \mathbb{N}$ and vectors $y_{1},...y_{L},x_{1},...,x_{N}\in X_{0}$, there exists an integer $n\geq K$ and a vector $z\in X$ with $\|z\|< \varepsilon$ such that}
\textnormal{
\[\|T^{n}_{i}y_{\ell}\|<\varepsilon \textrm{ and } \|T^{n}_{i}z-x_{i}\|<\varepsilon \textrm{ for integers } 1\leq i\leq N \textrm{ and } 1\leq \ell \leq L.\]
}
\end{remark}

We are now ready to discuss a sufficient condition for bilateral pseudo-shifts to be disjoint hypercyclic. In \cite{CMS20}, they characterized disjoint hypercyclic unilateral pseudo-shifts on $\ell^{p}(\mathbb{N})$. Here we only give a sufficient condition for bilateral pseudo-shifts on $\ell^{p}(\mathbb{Z})$ to be disjoint hypercyclic. Before giving the theorem, we provide a useful lemma to prove the theorem.

\begin{lemma}{\label{lemma2.3}}
Let $1\leq p <\infty$ and $N\geq2$. For each integer $i$ with $1\leq i \leq N$, let $T_{i}=T_{f_{i},w^{(i)}}$ be the bilateral pseudo-shifts on $\ell^{p}(\mathbb{Z})$ induced by the map $f_{i}: \mathbb{Z}\rightarrow \mathbb{Z}$ and the weight sequence $w^{(i)}=(w^{(i)}_{m})_{m\in \mathbb{Z}}$, and let $W^{(i)}_{m,n}=\prod^{n}_{v=1}w^{(i)}_{f^{v}_{i}(m)}$.
For integer $n,M\in \mathbb{N}$ and vectors $x_{1}$,...,$x_{N}$ given by $x_{i}=\sum^{M}_{m=-M}a^{(i)}_{m}e_{m}$ with $a^{(i)}_{-M}$,...,$a^{(i)}_{M}\neq 0$, there exists a vector $z\in$ span$\{e_{m}: m\in \mathbb{Z}\}$ such that

\[\|z\| \leq (2M+1)N\Gamma \max\{|W^{(\ell)}_{m,n}|: 1\leq \ell \leq N, 1\leq m\leq M\},\]
\[\|T^{n}_{1}z_{n}-x_{1}\| \leq (2M+1)N\Gamma \max\left\{\left|\frac{W^{(1)}_{f^{-n}_{1}(j),n}}{W^{(\ell)}_{f^{-n}_{\ell}(j),n}}\right|: 2\leq \ell \leq N, j\in f^{n}_{\ell}([M])\cap f^{n}_{1}(\mathbb{Z}\setminus[M])\right\},\]
and for each integer $i$ with $2\leq i\leq N$, we have

\[\|T^{n}_{i}z-x_{i}\|\leq (2M+1)N\Gamma\max\left\{\left|\frac{W^{(i)}_{f^{-n}_{i}(j),n}}{W^{(\ell)}_{f^{-n}_{\ell}(j),n}}-
\frac{a^{(i)}_{f^{-n}_{i}(j)}}{a^{(\ell)}_{f^{-n}_{\ell}(j)}}\right|:1\leq \ell \leq i-1, j\in f^{n}_{\ell}{([M])}\cap f^{n}_{i}{([M])}\right\}\]
\[+(2M+1)N\Gamma\max\left\{\left|\frac{W^{(i)}_{f^{-n}_{i}(j),n}}{W^{(\ell)}_{f^{-n}_{\ell}(j),n}}\right|: \ell\neq i, j\in f^{n}_{\ell}{([M])}\cap f^{n}_{i}{(\mathbb{Z}\setminus[M])}\right\},\]
where $\Gamma=\max\{\|x_{i}\|:1\leq i\leq N\}$, $W^{(i)}_{m,n}=\prod^{n}_{v=1}w^{(i)}_{f^{v}_{i}(m)}$ and $[M]=\{-M,...,0,...,M\}$ with $M\in\mathbb{N}$.
\end{lemma}

We will not give the proof of Lemma \ref{lemma2.3}, since it is very similar to that of Lemma 2.5 in \cite{CMS20}.

\begin{theorem}\label{them2.4}
Let $1\leq p <\infty$ and $N>2$. For each integer $i$ with $1\leq i \leq N$, let $T_{i}=T_{f_{i},w^{(i)}}$ be the bilateral pseudo-shifts on $\ell^{p}(\mathbb{Z})$ induced by the map $f_{i}: \mathbb{Z}\rightarrow \mathbb{Z}$ and the weight sequence $w^{(i)}=(w^{(i)}_{m})_{m\in \mathbb{Z}}$, and let $W^{(i)}_{m,n}=\prod^{n}_{v=1}w^{(i)}_{f^{v}_{i}(m)}$.
Suppose that there exists a strictly increasing sequence $(n_{k})_{k\in \mathbb{N}}$ of positive integers which satisfies the following:

$\;\;$        $(a)$ for each integer $i\in \mathbb{N}$ with $1\leq i\leq N$ and $m\in \mathbb{Z}$ we have $\left|W^{(i)}_{m, n_{k}}\right|\rightarrow \infty$, $\left|W^{(i)}_{m, -n_{k}}\right|\rightarrow 0$ as $k \rightarrow \infty$,  and

$\;\;$        $(b)$ for each $\varepsilon > 0$, integers $K$, $M\in \mathbb{N}$, and finite collection $\{a^{(i)}_{j}: -M\leq j \leq M, 1\leq i\leq N\}$ of non-zero scalars in $\mathbb{K}\setminus \{0\}$, there exists an integer $k\geq K$ such that for any integers $i, \ell$ with $1\leq i, \ell \leq N$ and $i\neq \ell$, we have

\[\left|\frac{W^{(i)}_{f^{-n_{k}}_{i}(j),n_{k}}}{W^{(\ell)}_{f^{-n_{k}}_{\ell}(j),n_{k}}}\right|< \varepsilon \textnormal{ whenever } j\in f^{n_{k}}_{\ell}([M])\cap f^{n_{k}}_{i}(\mathbb{Z}\setminus [M]), \]
and

\[\left|\frac{W^{(i)}_{f^{-n_{k}}_{i}(j),n_{k}}}{W^{(\ell)}_{f^{-n_{k}}_{\ell}(j),n_{k}}}-\frac{a^{(i)}_{f^{-n_{k}}_{i}(j)}}
{a^{(\ell)}_{f^{-n_{k}}_{\ell}(j)}}\right| < \varepsilon \textnormal{ whenever } j\in f^{n_{k}}_{\ell}([M])\cap f^{n_{k}}_{i}([M]),\]
where $[M]=\{-M,-M+1,...,0,...,M\}$.
Then the pseudo-shifts $T_{1}$,....$T_{N}$ satisfy the Disjoint Blow-up/Collapse Criterion. In particular, they are disjoint hypercyclic.
\end{theorem}

\begin{proof}
We prove this theorem in a similar way with the proof of Theorem 2.4 of \cite{CMS20}.
To show that the bilateral pseudo-shifts $T_{1},...,T_{N}$ satisfy the Disjoint Blow-up/Collapse Criterion, it is sufficient to prove that they satisfy the equivalent condition given in Remark \ref{rmk2.2}. As the proof of Theorem 2.4 of \cite{CMS20}, let $X_{0}=span\{e_{m}:m\in \mathbb{Z}\}$, which is dense in $\ell^{p}(\mathbb{Z})$. For any $\varepsilon >0$, integers $K, L\in \mathbb{N}$ and vectors $y_{1},...,y_{L}$,$x_{1},...,x_{N}\in X_{0}$, we can write each vector $x_{i}$ as

\[x_{i}=\sum^{M}_{m=-M}a^{(i)}_{m}e_{m}, \textrm{ where } M=\max\{deg(y_{\ell}),deg(x_{i}):1\leq \ell\leq L, 1\leq i\leq N\}. \]

Without loss of generality, we may assume $a^{(i)}_{-M},...,a^{(i)}_{M}\neq 0$ for integers $1\leq i\leq N$ by a similar argument of Theorem 2.4 of \cite{CMS20}.

Since for any integer $i\in \mathbb{N}$ with $1\leq i\leq N$ and $m\in \mathbb{Z}$, $\left|W^{(i)}_{m, -n_{k}}\right|\rightarrow 0$ as $k \rightarrow \infty$, for any $1\leq \ell \leq L$ and any $1\leq i\leq N$, $T^{n_{k}}_{i}y_{\ell}\rightarrow 0$ as $k \rightarrow \infty$.
Then there exists an integer $M'$ such that for any $k>M'$ and for any $1\leq \ell \leq L$ and $1\leq i\leq N$, $\|T^{n_{k}}_{i}y_{\ell}\|<\varepsilon$.

From the condition $(a)$ and $(b)$, there exists an integer $n=n_{k}$ with $k\geq \max\{K, M, M'\}$ such that

\[\left|W^{(i)}_{m,n}\right|> \frac{(2M+1)N\Gamma}{\varepsilon}, \textrm{ for integers } 1\leq i\leq N \textrm{ and } -M\leq m\leq M,\]

\[\left|\frac{W^{(i)}_{f^{-n}_{i}(j),n}}{W^{(\ell)}_{f^{-n}_{\ell}(j),n}}\right|< \frac{\varepsilon}{2(2M+1)N\Gamma}, \textrm{ whenever } i\neq \ell \textrm{ and } j\in f^{n}_{\ell}([M]) \cap f^{n}_{i}(\mathbb{Z}\setminus [M]),\]

\[\left|\frac{W^{(i)}_{f^{-n}_{i}(j),n}}{W^{(\ell)}_{f^{-n}_{\ell}(j),n}}- \frac{a^{(i)}_{f^{-n}_{i}(j)}}{a^{(\ell)}_{f^{-n}_{\ell}(j)}}\right|< \frac{\varepsilon}{2(2M+1)N\Gamma}, \textrm{ whenever } i\neq \ell \textrm{ and } j\in f^{n}_{\ell}([M]) \cap f^{n}_{i}([M]),\]
where $\Gamma=\max\{\|x_{i}\|:1\leq i\leq N\}$. By Lemma \ref{lemma2.3} with the integers $M$, $n=n_{k}$ and vectors $x_{1},...,x_{N}$, there is a vector $z\in X_{0}$ which satisfies the inequalities in Lemma \ref{lemma2.3}.
Then we get

\[\|z\| \leq (2M+1)N\Gamma \max\{|W^{(\ell)}_{m,n}|: 1\leq \ell \leq N, 1\leq m\leq M\}< \varepsilon,\]
\[\|T^{n}_{1}z-x_{1}\| \leq (2M+1)N\Gamma \max\left\{\left|\frac{W^{(1)}_{f^{-n}_{1}(j),n}}{W^{(\ell)}_{f^{-n}_{\ell}(j),n}}\right|: 2\leq \ell \leq N, j\in f^{n}_{\ell}([M])\cap f^{n}_{1}(\mathbb{Z}\setminus[M])\right\}<\varepsilon,\]
and for each integer $i$ with $2\leq i\leq N$, we have

\[\|T^{n}_{i}z-x_{i}\|\leq (2M+1)N\Gamma\max\left\{\left|\frac{W^{(i)}_{f^{-n}_{i}(j),n}}{W^{(\ell)}_{f^{-n}_{\ell}(j),n}}-
\frac{a^{(i)}_{f^{-n}_{1}(j)}}{a^{(\ell)}_{f^{-n}_{\ell}(j)}}\right|:1\leq \ell \leq i-1, j\in f^{n}_{\ell}{([M])}\cap f^{n}_{i}{([M])}\right\}\]
\[+(2M+1)N\Gamma\max\left\{\left|\frac{W^{(i)}_{f^{-n}_{i}(j),n}}{W^{(\ell)}_{f^{-n}_{\ell}(j),n}}\right|: \ell\neq i, j\in f^{n}_{\ell}{([M])}\cap f^{n}_{i}{(\mathbb{Z}\setminus[M])}\right\}<\frac{\varepsilon}{2}+\frac{\varepsilon}{2}<\varepsilon,\]
which concludes the proof.
\end{proof}

Using Theorem \ref{them2.4}, we can provide some examples of invertible pseudo-shifts on $\ell^{p}(\mathbb{Z})$, which are disjoint hypercylic.

\begin{example}\label{ex2.5}
Let define $T_{i}=T_{f_{i}, w^{(i)}}$ ($1\leq i\leq N$) on $\ell^{p}(\mathbb{Z})$  with $1\leq p <\infty$ as follows:
$f_{i}$ is given by $f_{i}(n)=n+p_{i}$, $n\in \mathbb{Z}$ where $p_{1},...,p_{N}$ are positive integers such that for every $1\leq s< t \leq N$, $2p_{s}<p_{t}$ and $w^{(i)}=(w^{(i)}_{n})_{n\in \mathbb{Z}}$ is given by
\begin{displaymath}
w^{(i)}_{n} = \left\{ \begin{array}{ll}
\lambda_{i} & \textrm{if $n>l_{i}$},\\
\frac{1}{\lambda_{i}} & \textrm{if $n\leq l_{i}$},
\end{array} \right.
1\leq i\leq N,
\end{displaymath}
where $l_{1},...,l_{N}$ are arbitrary integers and $\lambda_{1},...,\lambda_{N}$ are real numbers such that for every $1\leq s< t \leq N$, $1<|\lambda_{s}|<|\lambda_{t}|$.

From the definition of $T_{i}=T_{f_{i}, w^{(i)}}$ ($1\leq i\leq N$), we can easily check that $T_{i}=T_{f_{i}, w^{(i)}}$ ($1\leq i\leq N$)  are invertible operators.
In fact, setting $g_{i}(n)=n-p_{i}$ and $v^{(i)}_{n}=\frac{1}{w^{(i)}_{n+p_{i}}}$ gives $T^{-1}_{i}=T_{g_{i}, v^{(i)}}$ which is also a pseudo-shift.

Now we can see the fact that $T_{i}=T_{f_{i}, w^{(i)}}$ ($1\leq i\leq N$) are disjoint hypercyclic by showing that $T_{1},...,T_{N}$ satisfy the conditions of Theorem \ref{them2.4}.

We take the sequence $(k)_{k\in\mathbb{N}}$ as a increasing sequence $(n_{k})_{k\in \mathbb{N}}$.
Since $\left|W^{(i)}_{m, k}\right|\rightarrow \infty$, $\left|W^{(i)}_{m, -k}\right|\rightarrow 0$ as $k \rightarrow \infty$, the condition (a) of Theorem \ref{them2.4} is satisfied.
Now it is sufficient to prove that for any $\varepsilon > 0$, integers $K$, $M\in \mathbb{N}$, there exists an integer $k\geq K$ such that for any integers $i, \ell$ with $1\leq i, \ell \leq N$ and $i\neq \ell$, we have $f^{k}_{\ell}([M])\cap f^{k}_{i}([M])=\emptyset$ and

\[\left|\frac{W^{(i)}_{f^{-k}_{i}(j),k}}{W^{(\ell)}_{f^{-k}_{\ell}(j),k}}\right|< \varepsilon \textnormal{ whenever } j\in f^{k}_{\ell}([M])\cap f^{k}_{i}(\mathbb{Z}\setminus [M]). \]

Since $f^{k}_{i}([M])=\{-M+kp_{i},...,M+kp_{i}\}$, if $k>2M$ then $f^{k}_{\ell}([M])\cap f^{k}_{i}([M])=\emptyset$.

Let $\gamma=\max\{\frac{|\lambda_{s}|}{|\lambda_{s+1}|}: s<N\}$, $\alpha=\min\{|\lambda_{s}|:1\leq s\leq N\}$, $\beta=\max\{|\lambda_{s}|:1\leq s\leq N\}$ and $L=\max\{|l_{s}|:1\leq s\leq N\}$.
Then $\alpha, \beta>1$ and $\gamma<1$.

In the case of $\ell>i$, if $k>\frac{\ln\varepsilon-\ln\{\beta^{2\frac{M+L}{p_{1}}}\}}{\ln\gamma}$ and $t\in f^{k}_{\ell}([M])\cap  f^{k}_{i}(\mathbb{Z}\setminus [M])$ then
\[\left|\frac{W^{(i)}_{f^{-k}_{i}(t), k}}{W^{(\ell)}_{f^{-k}_{\ell}(t), k}}\right|=\left|\frac{\prod^{k}_{j=1}w^{(i)}_{t-kp_{i}+jp_{i}}}{\prod^{k}_{j=1}w^{(\ell)}_{t-kp_{\ell}+jp_{\ell}}}\right|\leq\left|\frac{\lambda_{i}^{k}}{\lambda_{\ell}^{k-2\frac{M+l_{\ell}}{p_{\ell}}}}\right|\leq \gamma^{k}\left|\lambda_{\ell}^{2\frac{M+l_{\ell}}{p_{\ell}}}\right|\leq\gamma^{k}\beta^{2\frac{M+L}{p_{1}}}<\varepsilon.\]
And in the case of $i>\ell$, if $k>\frac{-\ln{\varepsilon}+\ln\{\beta^{\frac{4M+4L}{p_{1}}}\}}{\ln\alpha}$ and $t\in f^{k}_{\ell}([M])\cap  f^{k}_{i}(\mathbb{Z}\setminus [M])$ then
\[\left|\frac{W^{(i)}_{f^{-k}_{i}(t),k}}{W^{(\ell)}_{f^{-k}_{\ell}(t),k}}\right|=
\left|\frac{\prod^{k}_{j=1}w^{(i)}_{t-kp_{i}+jp_{i}}}{\prod^{k}_{j=1}w^{(\ell)}_{t-kp_{\ell}+jp_{\ell}}}\right|
\leq\left|\frac{\lambda_{i}^{k-2\frac{l_{i}-\{M+kp_{\ell}-kp_{i}\}}{p_{i}}}}{\lambda_{\ell}^{k-2\frac{M+l_{\ell}}{p_{\ell}}}}\right|
=\]
\[=\left|\frac{\lambda_{i}^{k-2\frac{kp_{i}-kp_{\ell}-M+l_{i}}{p_{i}}}}{\lambda_{\ell}^{k-2\frac{M+l_{\ell}}{p_{\ell}}}}\right|
=\left|\frac{\lambda_{\ell}^{2\frac{M+l_{\ell}}{p_{\ell}}}\lambda_{i}^{2\frac{M-l_{i}}{p_{i}}}\lambda_{i}^{\frac{k(2p_{\ell}-p_{i})}{p_{i}}}}{\lambda_{\ell}^{k}}\right|
\leq\left|\frac{\lambda_{\ell}^{2\frac{M+l_{\ell}}{p_{\ell}}}\lambda_{i}^{2\frac{M-l_{i}}{p_{i}}}}{\lambda_{\ell}^{k}}\right|
\leq \frac{\beta^{\frac{4M+4L}{p_{1}}}}{\alpha^{k}}<\varepsilon.\]
So we get that the operators $T_{1},...,T_{N}$ are disjoint hypercyclic.
And similarly we can check $T^{-1}_{1},...T^{-1}_{N}$ are also disjoint hypercyclic.
\end{example}

Above the example, all the operators $T_{i}=T_{f_{i}, w^{(i)}}$ ($1\leq i\leq N$) can be described as powers of some bilateral weighted shifts. In the following example, we will provide bilateral pseudo-shifts, which one of them can not be decribed as powers of weighted shifts, but they are disjoint hypercyclic.
Here, we will discuss our result on the space $\ell^{p}(\mathbb{Z}\setminus\{0\})$ for the convenience of simbols.
 
\begin{example}
Let $T_{f_{1}, w^{(1)}}$, $T_{f_{2}, w^{(2)}}$ be the bilateral pseudo-shifts on $\ell^{p}(\mathbb{Z}\setminus{\{0\}})$ with $1\leq p <\infty$ induced by the invertible maps $f_{1},f_{2}:\mathbb{Z}\setminus{\{0\}}\rightarrow \mathbb{Z}\setminus{\{0\}}$ given by
\begin{displaymath}
f_{1}(i) = \left\{ \begin{array}{ll}
1 & \textrm{if $i=-1$},\\
i+1 & \textrm{if $i\neq -1$},
\end{array} \right.
\end{displaymath}
and $f_{2}(2^{i}(2j+1))=2^{i}(2j+3)$ for $i\geq 0$, $j\in\mathbb{Z}$
and weighted sequences $w^{(1)}$, $w^{(2)}$ given by
	
\begin{displaymath}
w^{(1)}_{i} = \left\{ \begin{array}{ll}
\lambda_{1} & \textrm{if $i>0$},\\
\frac{1}{\lambda_{1}} & \textrm{if $i<0$},
\end{array} \right.
\textrm{ and }
w^{(2)}_{i} = \left\{ \begin{array}{ll}
\lambda_{2} & \textrm{if $i>0$},\\
\frac{1}{\lambda_{2}} & \textrm{if $i<0$},
\end{array} \right.
\end{displaymath}
where $\lambda_{1}$, $\lambda_{2}$ are real numbers such that $|\lambda_{2}|>|\lambda_{1}|>1$.
	
Then we can easily check that $f_{1}$, $f_{2}$ are invertible on $\mathbb{Z}\setminus{\{0\}}$ and for every $m\in \mathbb{Z}\setminus{\{0\}}$, $\lim_{n\rightarrow+\infty}f^{-n}_{1}(m)=\lim_{n\rightarrow+\infty}f^{-n}_{2}(m)=-\infty$ and $\lim_{n\rightarrow+\infty}f^{n}_{1}(m)=\lim_{n\rightarrow+\infty}f^{n}_{2}(m)=+\infty$, which allows $\lim_{n \rightarrow +\infty}|W^{(1)}_{m,-n}|=\lim_{n \rightarrow +\infty}|W^{(2)}_{m,-n}|=0$ and $\lim_{n \rightarrow +\infty}|W^{(1)}_{m,n}|=\lim_{n \rightarrow +\infty}|W^{(2)}_{m,n}|=+\infty$
	
Let $\varepsilon>0$ be a real number and $K,M\in\mathbb{N}$.
Since $f_{2}(i)\geq i+2$ for every $i\in \mathbb{Z}\setminus{\{0\}}$, there exists a natural number $n_{1}\in \mathbb{N}$ such that for every $k>n_{1}$, $f^{k}_{1}([M])\cap f^{k}_{2}([M])=\emptyset$, where $[M]=\{-M,-M+1,...,-1,1,2,...,M\}$.
	
For $j\in f^{n}_{1}([M])$, we obtain $\left|W^{(1)}_{f^{-n}_{1}(j),n}\right|\geq \left|\lambda_{1}^{n}/\lambda_{1}^{-2M}\right|$ and $\left|W^{(2)}_{f^{-n}_{2}(j),n}\right|\leq \left|\lambda_{2}^{M+1}\right|$, which is allowed by the fact that $\sharp\{k | f^{-k}_{2}(j)>0, 0\leq k\leq n\} \leq (n+M)/2$ and $\sharp\{k|f^{-k}_{1}(j)<0, 0\leq k\leq n\}\leq M$, where $\sharp\{\cdotp\}$ is a cardinal number of $\{\cdotp\}$.  Then we get
\[\left|\frac{W^{(2)}_{f^{-n}_{2}(j),n}}{W^{(1)}_{f^{-n}_{1}(j),n}}\right|\leq\left|\frac{\lambda_{2}^{M+1}}{\lambda_{1}^{n-2M}}\right|.\]
And for $j\in f^{n}_{2}([M])$, $\left|W^{(2)}_{f^{-n}_{2}(j),n}\right|\geq \left|\lambda_{2}^{n}/\lambda_{2}^{-2M}\right|$ and $\left|W^{(1)}_{f^{-n}_{1}(j),n}\right|\leq \left|\lambda_{1}^{n}\right|$, then 
\[\left|\frac{W^{(1)}_{f^{-n}_{1}(j),n}}{W^{(2)}_{f^{-n}_{2}(j),n}}\right|\leq\left|\frac{\lambda_{1}^{n}}{\lambda_{2}^{n-2M}}\right|.\]
	
Since $\lim_{n \rightarrow \infty}\left|\frac{\lambda_{2}^{M+1}}{\lambda_{1}^{n-2M}}\right|=\lim_{n \rightarrow \infty}\left|\frac{\lambda_{1}^{n}}{\lambda_{2}^{n-2M}}\right|=0$, there exist a natural number $n_{2}\in \mathbb{N}$ such that $\left|\frac{\lambda_{2}^{M+1}}{\lambda_{1}^{n-2M}}\right|<\varepsilon$ and $\left|\frac{\lambda_{1}^{n}}{\lambda_{2}^{n-2M}}\right|<\varepsilon$.
	
Therefore for $k>\max\{K,n_{1},n_{2}\}$,
\[\left|\frac{W^{(2)}_{f^{-n}_{2}(j),n}}{W^{(1)}_{f^{-n}_{1}(j),n}}\right|<\varepsilon, \textrm{ for } j\in f^{n}_{1}([M]),\]
\[\left|\frac{W^{(1)}_{f^{-n}_{1}(j),n}}{W^{(2)}_{f^{-n}_{2}(j),n}}\right|<\varepsilon, \textrm{ for } j\in f^{n}_{2}([M])\]
and
\[ f^{n}_{1}([M])\cap f^{n}_{2}([M])=\emptyset.\]
	
From Theorem \ref{them2.4}, $T_{f_{1}, w^{(1)}}$ and $T_{f_{2}, w^{(2)}}$ are disjoint hypercyclic.
\end{example}

\section{Disjoint reiterative hypercyclicity}

In this section we discuss Question \ref{qest1.1}.  We give a negative answer to Question \ref{qest1.1} when $X$ is a reflexive Banach space.  First, we recall the following results.

\begin{proposition} \label{propo3.1}\textnormal{(\cite{GL22})}
Let $X$ be a reflexive Banach space and $T\in \mathcal{L}(X)$. Then $T$ is reiteratively recurrent if and only if $T$ admits a invariant probability measure with full support.
\end{proposition}

\begin{proposition} \label{propo3.2}\textnormal{(Theorem 2.1 of \cite{BGLP22})}
An operator $T\in \mathcal{L}(X)$ is reiteratively hypercyclic if and only if $T$ is hypercyclic and reiteratively recurrent.
\end{proposition}

\begin{proposition} \label{propo3.3}\textnormal{(Theorem 14 of \cite{BMPP16})}
If $T$ is reiteratively hypercyclic then its set of reiteratively hypercyclic vectors coincides with its set of hypercyclic vectors.
\end{proposition}

Next theorem shows that reiteratively hypercyclic operators on a reflexive Banach space are disjoint hypercyclic if and only if they are disjoint reiteratively hypercyclic, which gives a negative answer partially to Question \ref{qest1.1}.

\begin{theorem} \label{them3.2}

Let $X$ be a reflexive Banach space and $N\geq2$ be a positive integer, $T_{1},...,T_{N}$ be continuous linear operators on $X$.
Then the following assertions are equivalent.

$(1)$ $T_{1},...,T_{N}$ are disjoint hypercyclic and they are respectively reiteratively hypercyclic operators.

$(2)$ $T_{1},...,T_{N}$ are disjoint reiteratively hypercyclic.

\end{theorem}

\begin{proof}
$(2)\Rightarrow(1)$ is obvious.

To show that $(1)$ implies $(2)$, let $x$ be a disjoint hypercyclic vector.
First we can easily deduce that $T_{1},...,T_{N}$ admit invariant probability measures $\mu_{1},...,\mu_{N}$ with full support from Proposition \ref{propo3.1}.
So the operator $T_{1} \oplus\cdots\oplus T_{N}$ on $X^{N}$ admits a invariant probability measure with full support.

Since $X^{N}$ is a reflexive space and $T_{1} \oplus\cdots\oplus T_{N}$ admits a invariant probability measure with full support, then $T_{1}\oplus\cdots\oplus T_{N}$ is reiteratively recurrent and therefore it is reiteratively hypercyclic on $X^{N}$ from Proposition \ref{propo3.2}.
From Proposition \ref{propo3.3}, ($x$,...,$x$) is a reiteratively hypercyclic vector and thus $T_{1}$,..., $T_{N}$ are disjoint reiteratively hypercyclic which concludes the proof.
\end{proof}

Martin, Menet and  Puig \cite{MMP22} showed that for two reiteratively hypercyclic unilateral pseudo-shifts  on $\ell^{p}(\mathbb{N})$, $1\leq p <\infty$, if they have the same inducing map, they are disjoint hypercyclic if and only if they are disjoint reiteratively hypercyclic (Theorem 4.3 in \cite{MMP22}). Using Theorem \ref{them3.2}, we know if $1< p <\infty$, then this fact hold without the assuming that they have the same inducing map as follows:

\begin{corollary}\label{cor3.1}
Let $X=\ell^{p}(\mathbb{N})$ ($\ell^{p}(\mathbb{Z})$), $1< p <\infty$ and $N>2$. For each integer $i$ with $1\leq i \leq N$, let $T_{i}=T_{f_{i},w_{i}}$ be unilateral (bilateral) pseudo-shifts on $X$. Then $T_{1},...,T_{N}$ are disjoint reiteratively hypercyclic if and only if $T_{1},...,T_{N}$ are reiteratively hypercyclic and disjoint hypercyclic.
\end{corollary}


\begin{thebibliography}{99}

\bibitem{B07}
\newblock L. Bernal-Gonz\'{a}lez,
\newblock Disjoint hypercyclic operators,
\newblock Studia Math., \textbf{182} (2), 113-130, 2007.

\bibitem{BMS14}
\newblock J. B\`{e}s, \"{O}. Martin and R. Sanders,
\newblock Weighted shifts and disjoint hypercyclicity,
\newblock  J. Operator Theory, \textbf{72} (1), 15-40, 2014.

\bibitem{BMPP16}
\newblock J. B\`{e}s, Q. Menet, A. Peris, Y. Puig,
\newblock Recurrence properties of hypercyclic operators,
\newblock Math. Ann., 366:545-572, 2016.

\bibitem{BGLP22}
\newblock J. B\`{e}s, K.-G. Grosse-Erdmann, A. L\'{o}pez-Mart\'{\i}nez A. Peris,
\newblock Frequently recurrent operators,
\newblock J. Funct. Anal., \textbf{283}(12), 109713, 2022.

\bibitem{BP07}
\newblock J. B\`{e}s and A. Peris,
\newblock Disjointness in hypercyclicity,
\newblock J. Math. Anal. Appl., \textbf{336}, 297-315, 2007.

\bibitem{CMS20}
\newblock N. \c{C}olako\u{g}lu, \"{O}. Martin and R. Sanders,
\newblock Disjoint and simultaneous hypercyclic pseudo-shift operators,
\newblock arXiv:2112.04884v1 [math.FA] 9 Dec 2021.(J. Math. Anal. Appl.,\textbf{512}, No. 2 (August 2022))

\bibitem{GL22}
\newblock S. Grivaux, A. L\'{o}pez-Mart\'{\i}nez,
\newblock Recurrence properties for linear dynamical systems: An approach via invariant measures,
\newblock J. Math. Pures Appl., \textbf{169}, no. 9,  155-188, 2022.

\bibitem{G00}
\newblock K.-G. Grosse-Erdmann,
\newblock Hypercyclic and chaotic weighted shifts,
\newblock Studia Math., \textbf{139} (1), 47-68, 2000.

\bibitem{MMP22}
\newblock \"{O}. Martin, Q. Menet and Y. Puig,
\newblock Disjoint frequently hypercyclic pseudo-shifts,
\newblock J. Funct. Anal., \textbf{283} (1), 109474, 2022.


\bibitem{MP16}
\newblock \"{O}. Martin and Y. Puig,
\newblock Existence of disjoint frequently hypercyclic operators which fail to be disjoint weakly mixing,
\newblock J. Math. Anal. Appl., \textbf{500}, 125106, 2021.

\bibitem{S13}
\newblock H. Salas,
\newblock The strong disjoint blow-up/collapse property,
\newblock J. Funct. Spaces Appl., \textbf{2013}, Article ID 146517, 6 pages, 2013.

\bibitem{SS14}
\newblock R. Sanders and S. Shkarin,
\newblock Existence of disjoint weakly mixing operators that fail to satisfy the Disjoint Hypercyclicity Criterion,
\newblock J. Math. Anal. Appl., \textbf{417}, 834-855, 2014.

\bibitem{S10}
\newblock S. Shkarin,
\newblock A short proof of existence of disjoint hypercyclic operators,
\newblock J. Math. Anal. Appl., \textbf{367} (2), 713-715, 2010.









\end{thebibliography}
\end{document}